\documentclass[11pt]{amsart}
\usepackage{verbatim, latexsym, amssymb, amsmath,color}
\usepackage{stmaryrd}
\usepackage{amsfonts}
\usepackage{bbm}
\usepackage{amscd}
\usepackage{mathrsfs}
\usepackage{latexsym,amssymb,amsmath,amscd,amscd,amsthm,amsxtra}
\usepackage[dvips]{graphicx}
\usepackage[utf8]{inputenc}
\usepackage[T1]{fontenc}
\usepackage{lmodern}
\usepackage{amssymb}
\usepackage[all]{xy}
\usepackage{nicefrac,mathtools,enumitem}
\usepackage{microtype}
\usepackage{changepage}
\usepackage[lite]{amsrefs}
\usepackage{fancyhdr}
\usepackage{hyperref}
\usepackage{cancel} %for editing
\usepackage[normalem]{ulem} %for editing

\def\R{\mathbb R}

\def\N{\mathbb N}

\def\Z{\mathcal Z}

\def\F{\mathbf F}
\def\n{\mathbf n}

\def\m{\mathbf m}

\def\dmn{\mathrm{dmn}}

\def\ric{\mathrm{Ric}}
\def\area{\mathrm{Area}}
\def\d{\mathrm{div}}
\def\dist{\mathrm{dist}}

\newtheorem{theorem}{Theorem}[section]
\newtheorem{lemma}[theorem]{Lemma}
\newtheorem{corollary}[theorem]{Corollary}
\newtheorem{proposition}[theorem]{Proposition}
\theoremstyle{remark}
\newtheorem{rmk}{Remark}
\newtheorem{example}[theorem]{Example}
\newtheorem{claim}[]{Claim}
\theoremstyle{definition}
\newtheorem{definition}[theorem]{Definition}

\numberwithin{equation}{section}

\title[Min-max minimal hypersurfaces with free boundary]{Min-max minimal hypersurface in manifolds with convex boundary and ${\mathbf {\emph \rm{Ric}}\geq 0}$}
\author{Zhichao Wang}
\date{\today}
\address{School of Mathematical Sciences, Peking University
Yiheyuan Road 5, Beijing, P.R.China, 100871}
\email{wangzhichank@gmail.com}

\thanks{}

\begin{document}
\begin{abstract}
    Let $(M^{n+1},\partial M,g)$ be a compact manifold with non-negative Ricci curvature, convex boundary and $2\leq n\leq 6$. We show that the min-max minimal hypersurface with respect to one-parameter families of hypersurfaces in $(M,\partial M)$ is orientable, of index one and multiplicity one.
\end{abstract}

\maketitle
\section{Introduction}
In 1960s, Almgren \cites{Alm62, Alm65} initiated a variational theory to find minimal submanifolds. In those papers, he also conjectured that the Morse index of min-max solution is bounded by the number of parameters. There have been tremendous understanding of this conjecture in closed manifolds \cite{MN16}. However, the general index bounds for free boundary min-max minimal hypersurfaces still remains open in compact manifolds with non-empty boundary. In this paper, we address this problem for compact manifolds with certain natural convexity assumptions.
\begin{theorem}\label{main thm}
Let $(M^{n+1},\partial M,g)$ be any connected, compact, orientable manifold with convex boundary, non-negative Ricci curvature and $2\leq n\leq 6$. Then the min-max minimal hypersurface $\Sigma$ corresponding to the fundamental class $[M]$ is orientable of multiplicity one, index one and has least area among all embedded orientable free boundary minimal hypersurfaces.
\end{theorem}
\begin{rmk}\label{remark1}
Under our assumptions of $(M,\partial M)$ in Theorem \ref{main thm}, there are no closed minimal hypersurfaces in $(M,\partial M)$ (see Appendix A). Hence all critical hypersurfaces of the area functional are minimal hypersurfaces with non-empty free boundary.
\end{rmk}
\begin{rmk}
Manifolds with positive Ricci curvature has been studied in a lot of papers \cites{FL, MN12, MN17, Zho, Zho17}. The remarkable results by Marques-Neves \cite{MN17} and Li-Zhou \cites{LZ} said that there are infinity many free boundary minimal hypersurfaces in these manifolds. In this paper, the conditions of non-negative Ricci curvature and the convex boundary are used to show:
\begin{itemize}
  \item the non-existence of the two-sided stable free boundary minimal hypersurfaces;
  \item the fact that any two immersed free boundary minimal hypersurfaces must intersect;
  \item the existence of local foliation with non-negative mean curvature;
  \item the second variation of the free boundary minimal hypersurface along the unit normal vector field to be negative, which would be crucial for ruling out the non-orientable case.
\end{itemize}
\end{rmk}
\begin{rmk}
As another surprising fact, we obtain the existence of a least area guy among all free boundary minimal hypersurface. This would follow straightforwardly if one had smooth compactness among all free boundary minimal hypersurfaces; however, all known compactness results require additional assumptions. For instance, Ambrozio-Carlotto-Sharp \cite[Theorem 2]{ACS} established the compactness of the free boundary minimal hypersurfaces in above settings under the condition of the first eigenvalues to be bounded. In dimension 3, Fraser-Li \cite[Theorem 1.2]{FL} proved the compactness of the space of the compact, properly embedded, free boundary minimal surfaces with fixed topology.
\end{rmk}

\vspace{1em}
Minimal submanifolds play important roles in mathematics for a long time, as they appear in a wide range of fields. However, the existence of minimal submanifolds puzzled mathematicians for hundreds of years. Before Almgren \cites{Alm62, Alm65}, mathematicians always need some topological constraints to show the existence of minimal surfaces. Almgren initiated a variational theory to find minimal submanifolds in any compact manifolds. Using this theory, he could prove the existence of a weak solution (as stationary varifold). For a closed manifold $M^{n+1}$, the regularity of the submanifold was improved by Pitts \cite{Pit} for $n\leq 5$, and Schoen-Simon \cite{SS} for $n=6$.

In compact manifolds with boundary, Gr\"{u}ter-Jost \cite{GJ}, De Lellis-Ramic \cite{DR} established the regularity for the free boundary problem when the boundary is convex. Li-Zhou \cite{LZ} proved the general regularity theorem for any compact manifold with boundary.

\vspace{0.5em}
It is also very natural to study the geometric properties of the min-max minimal hypersurfaces. For three-manifolds, Pitts-Rubinstein conjectured that the min-max minimal surface from one-parameter families should have index less than or equal to one. For any $S^3$ with positive Ricci curvature, Marques-Neves \cite{MN12} studied the min-max minimal surfaces from one-parameter families and obtained rigidity results. For general closed manifolds with positive Ricci curvature, Zhou \cite{Zho} proved the index bounds for min-max minimal hypersurfaces from one-parameter families in manifolds with positive Ricci curvature. Zhou \cite{Zho17} also characterized the min-max minimal hypersurfaces from one-parameter families for closed manifolds with high dimensions. Ketover-Marques-Neves \cite[Theorem 2.4]{KMN} improved Zhou's results by the Catenoid estimates. Related results have also been proved for the least area closed minimal hypersurfaces by Mazet-Rosenberg \cite{MR} and Song \cite{Son}.

In the theory of minimal surfaces, Morse index always provides a useful way to show the rigidity of the minimal hypersurfaces. In the proof of Willmore conjecture \cite{MN14}, Marques-Neves proved that the min-max minimal surface in $S^3$ from the canonical 5-parameter families has index 5, and then must be the Clifford torus by Urbano \cite{Urb}. More interesting relations between Morse index and topology of the minimal hypersurfaces have been obtained by \cites{Ros, CM, ACS16, ACS16_2}. It is also very interesting to know how large the Morse index of the hypersurfaces from the k-parameter families can be. Marques-Neves \cite{MN16} proved that the index $\leq k$ for any min-max minimal hypersurfaces from $k$-parameter families. Our results in this paper imply the existence of the index one free boundary minimal hypersurfaces in the manifolds under our assumptions.

\vspace{1em}
The main idea is as follows. The first part is inspired by Marques-Neves \cite[Theorem 2.1]{MN12} and Zhou \cite[Theorem 1.1]{Zho}. Given any compact manifold $(M,\partial M,g)$ as in Theorem \ref{main thm}, we first embed each free boundary minimal hypersurface $\Sigma$ into a good one-parameter family of hypersurfaces. To do this, we show the existence of a good local foliation around $\Sigma$ and then extend the foliation to be a sweepout. Comparing to the closed case, here we can not use the exponential map to construct the foliation since the exponential map is not well-defined near boundary. Instead, we use  the level sets of the distance function to the hypersurface  as the local foliation. Here a new free variation formula in \S \ref{two variations} is essentially used. In order to show this local foliation could be extended, we use a contradiction argument; if not true, the continuous min-max theory by De Lellis-Ramic \cite{DR} adapted to a half space would give another free boundary minimal hypersurface. We will reach a contradiction with Frankel's property if the new free boundary minimal hypersurface won't intersect $\Sigma$. This non-intersecting property follows by using $\Sigma$ as a barrier. This means that each foliation could be extended to be a good sweepout.

Next we would like to discretize all the families in order to use the discrete Almgren-Pitts theory. This follows from Li-Zhou \cite{LZ} directly. Then by the free boundary min-max theory \cite[Theorem 5.21, Theorem 6.2]{LZ}, we can get a free boundary minimal hypersurface with least area, which may be orientable of multiplicity one or non-orientable. To rule out the non-orientable case, we first show that the multiplicity of non-orientable part must be even. Recall the construction of sweepouts from non-orientable hypersurface, we show the multiplicity of non-orientable min-max hypersurface is exactly two. If the non-orientable case happened, inspired by the work of Ketover-Marques-Neves, we can add a cylinder (catenoid for $n=2$) to the sweepout to reduce the maximal area among all hypersurfaces, and hence get a contradiction. The key point here is the area expansion has non-zero second order term by the free variation formula in \S \ref{two variations}.

\vspace{1em}
The paper is organized as follows. In Section \ref{Preliminary}, we first derive two free variation formulas and use them to show the existence of good neighborhood of free boundary minimal hypersurfaces. In Section \ref{continuous construction}, we construct sweepouts from any free boundary minimal hypersurfaces in the continuous settings. In Section \ref{discrete settings}, we introduce the Almgren-Pitts min-max theory for compact manifolds with boundary, which is developed by Li-Zhou \cite{LZ}. In Section \ref{Discretization}, the sweepouts from Section 3 will be discretized to be continuous in the mass norm. Then we will show all the discretized sweepouts are in the homotopy class corresponding to the fundamental class (Theorem \ref{homotopy}). In Section \ref{Multiplicity and Orientation}, we characterize the multiplicity and orientation of the min-max hypersurfaces. Finally, we prove our main result in Section \ref{main proof}.

\vspace{1em}
{\bf Acknowledgements:} The author would like to thank his advisor Prof. Gang Tian for constant encouragement and support. The author would also like to thank Prof. Xin Zhou for suggesting the problem and many helpful discussions. This work was done while the author was visiting the Department of Mathematics at MIT, supported by China Scholarship Council (File No. 201606100023). The author would like to thank Prof. Bill Minicozzi for encouragement and Xiaomeng Xu for the help in writing. The author would also like to thank the Department of Mathematics at MIT for its hospitality and for providing a good academic environment.

\section{Preliminaries}\label{Preliminary}
For any hypersurface in some closed manifold, one can always obtain variation from isotopy of the ambient manifold. In $(M,\partial M,g)$, a compact manifold with boundary, we can also get variation of hypersurface $\Sigma$ from isotopy of $(M,\partial M,g)$. In this case, the vector field corresponding to the isotopy need to satisfies $X|_{\partial M}\in T(\partial M)$, denoted by $\mathfrak X(M,\partial M)$. However, there are some variations which are not easy to find the vector fields.
\begin{example}
Let $(M,\partial M,g)$ be some convex ball in $\R^3$, and $\Sigma$ be the intersection of $M$ and some plane $P\subset\R^3$. Then $\Sigma$ separates $M$ into two pieces, called $M^+$ and $M^-$. Let $r$ be the distance function to $\Sigma$ with $r|_{M^+}\geq 0$ and $r|_{M^-}\leq 0$. Then $\{r^{-1}(t)\}_{t\in(-1,1)}$ is a variation of $\Sigma$. However, it is difficult find the variation vector field in $\mathfrak X(M,\partial M)$. On the other hand,
\begin{equation}
\Sigma_t=L_t(P)\cap M,
\end{equation}
where $L_t$ is the parallel moving by constant vector field.
\end{example}

In this section, we study the area of this kind of variation, which will be used later.

\subsection{Two Variation Formulas}\label{two variations}
Let $(M,\partial M,g)$ be some compact manifold with boundary. We always embed $(M,\partial M, g)$ into some closed manifold $(\tilde M,\tilde g)$. Let $X$ be a vector field on $\tilde M$. Then there exists a family of diffeomorphisms $(F_t)_{0\leq t\leq 1}$ generated by $X$. For any hypersurface $\tilde\Sigma$, $F_t(\tilde\Sigma)$ is a hypersurface for $t$ small enough. Set
\begin{gather*}
    \tilde\Sigma_t=F_t(\tilde\Sigma),\\
    \Sigma=\tilde\Sigma\cap M,\\
    \Sigma_t=F_t(\tilde\Sigma)\cap M.
\end{gather*}
We call $\{\Sigma_t\}$ a \emph{free variation of $\Sigma$ by} $X$.  In this part, we show
\begin{lemma}[The first free variation formula]\label{first variation}
\begin{equation}
\frac{d}{dt}\area(\Sigma_t)=\int_{\Sigma_t} H\langle \n,X\rangle+\int_{\partial\Sigma_t}\big(\langle X,\nu\rangle-\frac{\langle X,\nu_{\partial M}\rangle}{|\nu_{\partial M}^\top|}\big),
\end{equation}
where $\n$ is the unit normal vector field on $\Sigma_t$, $H$ is the mean curvature of $\Sigma_t$ in $M$, $\nu$ is the outward normal vector field of $\partial\Sigma_t$ in $M$, $\nu_{\partial M}$ is the outward unit normal vector field of $\partial M\subset M$, and $\nu_{\partial M}^\top$ is the projection to $T(\partial\Sigma_t)$.
\end{lemma}
\begin{rmk}
Here $X|_{\partial M}$ may not be in $T(\partial M)$.
\end{rmk}

\begin{proof}[Proof of Lemma \ref{first variation}]
Let $r$ be the distance function to $\partial M$ such that $r(x)<0$ for $x\in M$ and $r(x)>0$ if $x\notin M\cup\partial M$. Then $\nabla r|_{\partial M}=\nu_{\partial M}$. Now take a non-increasing cut-off function $\phi:(-\infty,+\infty)\rightarrow[0,1]$ satisfying
\begin{equation}
\phi(r)=\left\{\begin{aligned}
        &1, \mathrm{\ if} &r\leq 0\\
        &0, \mathrm{\ if} &r\geq 1
        \end{aligned}\right..
\end{equation}
Set
\begin{equation}
    \varphi(s,r)=\phi(\frac{r}{s}).
\end{equation}
Then for any $\tilde\Sigma_t$, set
\begin{equation}\label{def of I}
    I(s,t)=\int_{\tilde\Sigma_t}\varphi(s,r)d\mu(x).
\end{equation}
Then we have
\begin{align*}
\frac{\partial}{\partial t}I(s,t)&=\int_{\tilde\Sigma_t}\varphi\d X+\int_{\tilde\Sigma_t}\frac{1}{s}\phi'(\frac{r}{s})\langle X,\nabla r\rangle\\
                                 &=\int_{\tilde\Sigma_t}\varphi\d X+\int_{\tilde\Sigma_t\cap(\mathcal B_s\setminus\mathcal B_0)}\frac{1}{s}\phi'(\frac{r}{s})\langle X,\nabla r\rangle\\
                                 &=\int_{\tilde\Sigma_t}\varphi\d X+\int_0^s\big(\int_{\tilde\Sigma_t\cap\partial\mathcal B_r}\frac{1}{s}\phi'(\frac{r}{s})\frac{\langle X,\nabla r\rangle}{|(\nabla r)^\top|}\big)dr,
\end{align*}
where $\mathcal B_s=\{x\in\tilde M:r(x)<s\}$. In the last equality, we used the Co-area Formula \cite[\S 12.7]{Sim}. Now for any $\epsilon>0$, $t\geq0$, there exists $\delta=\delta(\tilde M,\tilde\Sigma,X,\epsilon,t)>0$, such that for any $t',r$ satisfying $|t'-t|+r\leq\delta$,
\begin{gather*}
    \Big|\int_{\tilde\Sigma_{t'}\cap\partial\mathcal B_r}\frac{\langle X,\nabla r\rangle}{|(\nabla r)^\top|}-\int_{\partial\Sigma_t}\frac{\langle X,\nabla r\rangle}{|(\nabla r)^\top|}\Big|<\epsilon,\\
    \Big|\int_{\Sigma_{t'}}\varphi\d X-\int_{\Sigma_t}\varphi\d X\Big|<\epsilon.
\end{gather*}
For any $t,\alpha,s$, there exist $t'\in[t,t+\alpha]$, which may rely on $s$, such that
\begin{align*}
    \ &\Big|\frac{I(s,t+\alpha)-I(s,t)}{\alpha}-\int_{\Sigma_t}\varphi(s,t)\d X+\int_{\partial\Sigma_t}\frac{\langle X,\nabla r\rangle}{|(\nabla r)^\top|}\Big|\\
    =&\Big|\frac{\partial}{\partial t}I(s,t')-\int_{\Sigma_t}\varphi(s,t)\d X+\int_{\partial\Sigma_t}\frac{\langle X,\nabla r\rangle}{|(\nabla r)^\top|}\Big|\\
    \leq&\Big|\int_0^s\Big(\frac{1}{s}\phi'(\frac{r}{s})\big(\int_{\tilde\Sigma_{t'}\cap\partial\mathcal B_r}\frac{\langle X,\nabla r\rangle}{|(\nabla r)^\top|}-\int_{\partial\Sigma_t}\frac{\langle X,\nabla r\rangle}{|(\nabla r)^\top|}\big)\Big)\Big|\\
    \ \ &+\Big|\int_{\Sigma_{t'}}\varphi\d X-\int_{\Sigma_t}\varphi\d X\Big|
\end{align*}
Hence if we choose $|\alpha|+s\leq\delta$,
\begin{equation}\label{uniform estimate}
    \Big|\frac{I(s,t+\alpha)-I(s,t)}{\alpha}-\int_{\Sigma_t}\varphi(s,t)\d X+\int_{\partial\Sigma_t}\frac{\langle X,\nabla r\rangle}{|(\nabla r)^\top|}\Big|\leq 2\epsilon.
\end{equation}
Here we use $\phi'\leq 0$ and $\int_0^s\frac{1}{s}\phi'(\frac{r}{s})dr=-1$. Recall the definition of $I(s,t)$ (\ref{def of I}), we have
\begin{equation}
\lim_{s\rightarrow 0}I(s,t)=\area(\Sigma_t).
\end{equation}
Hence
\begin{equation}
    \frac{d}{dt}\area(\Sigma_t)=\lim_{\alpha\rightarrow 0}\lim_{s\rightarrow 0}\frac{I(s,t+\alpha)-I(s,t)}{\alpha}.
\end{equation}
By (\ref{uniform estimate}),
\begin{align*}
    \frac{d}{dt}\area(\Sigma_t)&=\lim_{s\rightarrow 0}\int_{\tilde\Sigma_t}\varphi\d X-\int_{\partial\Sigma_t}\frac{\langle X,\nabla r\rangle}{|(\nabla r)^\top|}\\
                               &=\int_{\Sigma_t}\d X-\int_{\partial\Sigma_t}\frac{\langle X,\nabla r\rangle}{|(\nabla r)^\top|}\\
                               &=\int_{\Sigma_t}H\langle X,\n\rangle+\int_{\partial\Sigma_t}\big(\langle X,\nu\rangle-\frac{\langle X,\nu_{\partial M}\rangle}{|\nu_{\partial M}^\top|}\big).
\end{align*}
\end{proof}

\begin{rmk}
Since we can choose the $\tilde M$ freely, it is no meaningful to discuss the parallel variation. That is, we only consider the $X|_\Sigma=f\n$. In that case,
\begin{equation}
\frac{d}{dt}\area(\Sigma_t)=\int_{\Sigma_t} H\langle \n,X\rangle-\int_{\partial\Sigma_t}\frac{\langle X,\nu_{\partial M}\rangle}{|\nu_{\partial M}^\top|},
\end{equation}
and the critical manifold is also the free boundary (possible empty) minimal hypersurface.
\end{rmk}
\begin{lemma}[The second free variation formula]\label{second variation}
In the case of free boundary minimal hypersurface and $X|_{\Sigma}=f\n$, the second variation is
\begin{equation}
    \frac{d^2}{dt^2}\area(\Sigma_t)\Big|_{t=0}=\int_{\Sigma}|\nabla f|^2-|A|^2f^2-\ric(X,X)-\int_{\partial\Sigma}h^{\partial\Sigma}(X,X).
\end{equation}
\end{lemma}
\begin{proof}
Set
\begin{equation}
    J(s,t):=\int_{\tilde\Sigma_t}\varphi\d X-\int_{\partial\Sigma_t}\frac{\langle X,\nabla r\rangle}{|(\nabla r)^\top|}.
\end{equation}
Hence
\begin{align*}
    \frac{\partial}{\partial t}J(s,0)&=\int_{\tilde\Sigma}\varphi\cdot(\d X)^2+\frac{\langle X,\nabla r\rangle}{s}\phi'(\frac{r}{s})\d X+\varphi\frac{d}{dt}(\d X) \\
    &\ -\int_{\partial\Sigma}\frac{d}{dt}\Big|_{t=0}\frac{\langle X,\nabla r\rangle}{|(\nabla r)^\top|}\\
    &=\int_{\tilde\Sigma}\varphi\cdot(\d X)^2+\varphi\frac{d}{dt}(\d X)-\int_{\partial\Sigma}\frac{d}{dt}\Big|_{t=0}\frac{\langle X,\nabla r\rangle}{|(\nabla r)^\top|} \\
    &\ +\int_0^s\int_{\partial\Sigma\cap(\mathcal B_s\setminus \mathcal B_0)}\frac{1}{s}\phi'(\frac{r}{s})\frac{\langle X,\nabla r\rangle}{|(\nabla r)^{\top}|}\d X
\end{align*}
Since $X$ is the normal variation vector field of $\tilde\Sigma$, and $\Sigma$ is a free boundary minimal hypersurface of $M$,
\begin{gather}
    \d X|_{\Sigma}=0,\\
    \langle X,\nabla r\rangle|_{\partial\Sigma}=0.
\end{gather}
These imply
\begin{align*}
    \lim_{s\rightarrow 0}\frac{\partial}{\partial t}J(s,0)&=\int_{\Sigma}\nabla_X(\d X)-\int_{\partial\Sigma}\nabla_X\langle X,\nabla r\rangle\\
                                                          &=\int_{\Sigma}|\nabla f|^2-|A|^2f^2-\ric(X,X)+\int_{\partial\Sigma}\langle\nabla_XX,\nu\rangle\\
                                                          &\ -\int_{\partial\Sigma}\nabla_X\langle X,\nabla r\rangle\\
                                                          &=\int_{\Sigma}|\nabla f|^2-|A|^2f^2-\ric(X,X)-\int_{\partial\Sigma}h^{\partial\Sigma}(X,X),
\end{align*}
where $h^{\partial M}$ is the second fundamental form of $\partial M$ with normal vector field $\nu_{\partial M}$. Next use the same argument in Lemma \ref{first variation}, we can show that
\begin{equation}
\frac{d^2}{dt^2}\area(\Sigma_t)\Big|_{t=0}=\lim_{s\rightarrow 0}\frac{\partial}{\partial t}J(s,0).
\end{equation}
This completes our proof.
\end{proof}

\begin{corollary}\label{variation}
Let $\Sigma$ be a free boundary minimal hypersurface of $M$, and $d$ is the distance function to $\Sigma$, then the first free variation of $\Sigma$ by $\nabla d$ is
\begin{equation}
    \frac{d}{dt}\area(d^{-1}(t))=\int_{\Sigma_t} H-\int_{\partial\Sigma_t}\frac{\langle \n,\nu_{\partial M}\rangle}{|\nu_{\partial M}^\top|},
\end{equation}
and the second free variation is
\begin{equation}
    \frac{d^2}{dt^2}\area(d^{-1}(t))\Big|_{t=0}=\int_{\Sigma}-|A|^2-\ric(\n,\n)-\int_{\partial\Sigma}h^{\partial\Sigma}(\n,\n).
\end{equation}
\end{corollary}

\subsection{Good Neighborhoods}
Now we embed $(M,\partial M,g)$ into some closed manifold $(\tilde M,\tilde g)$ with same dimension. Let $\tilde\Sigma$ be the extended closed hypersurface in $\tilde M$. Here we show that there exists a good neighborhood foliated by hypersurfaces with non-negative mean curvature. If $\Sigma$ is orientable, we first claim that
\begin{claim}\label{seperate}
$\Sigma$ separates $M$ into two pieces.
\end{claim}
\begin{proof}
If not, $\Sigma$ would be an non-trivial element in $H_n(M,\partial M,\mathbb Z)$. Now taking the area minimizing hypersurface in this homology class, we can obtain a stable minimal hypersurfaces with free boundary, which contradicts with the assumptions.
\end{proof}

\begin{lemma}\label{dist fctn}
Let $(\Sigma,\partial\Sigma)\subseteq(M,\partial M)$ be a free boundary minimal hypersurface, there exists a relative open set $U\supseteq \Sigma$ and $a>0$ such that
\begin{enumerate}
  \item $U$ is homomorphic to $\Sigma\times (-2a,2a)$;
  \item The two distance functions $d=\dist(\cdot,\Sigma)$ on $M$ and $\tilde d=\dist(\cdot, \tilde\Sigma)$ on $\tilde M$ are smooth on $U$. Moreover, $d|_U=\tilde d|_U$;
  \item The level set $d^{-1}(t)$ is a smooth hypersurface with boundary for $t\in(-2a,2a)$. Moreover, the level set with normal vector field $-\nabla d$ has non-negative mean curvature if $t\in(0,2a)$;
  \item $\area(d^{-1}(t_2))\leq \area(d^{-1}(t_1))$ if $0\leq t_1\leq t_2\leq 2a$.
\end{enumerate}
\end{lemma}
\begin{proof}
By Claim \ref{seperate}, $\Sigma$ separates $M$ into two pieces: $M^+$ has inward normal vector field $\n$ on $\Sigma$ and $M^-$ has outward normal vector field $\n$ on $\Sigma$.  We can define the distance function
\begin{equation}
d(x)=\dist(x,\Sigma) \text{ if } x\in M^+,
\end{equation}
and
\begin{equation}
d(x)=-\dist(x,\Sigma) \text{ if } x\in M^-.
\end{equation}
Take $U=d^{-1}(-2a,2a)$ and one can always shrink $a$ such that $(1)(2)$ and the first half of $(3)$ satisfied. Since $\Delta r\leq 0$, and the mean curvature of $d^{-1}(t)$ satisfies
\begin{equation}
H=\d(-\nabla d)=-\Delta d\geq 0.
\end{equation}
For (4), recall the first variation formula in Corollary \ref{variation},
\begin{equation}
    \frac{d}{dt}\area(d^{-1}(t))=-\int_{\Sigma_t} H-\int_{\partial\Sigma_t}\frac{\langle \nabla d,\nu_{\partial M}\rangle}{|\nu_{\partial M}^\top|},
\end{equation}
Since $\partial M$ is convex, we can shrink $a$ such that
\begin{equation}
    \langle \nabla d,\nu_{\partial M}\rangle\geq 0,
\end{equation}
then
\begin{equation}
\frac{d}{dt}\area(d^{-1}(t))\leq 0 \mathrm{\ for\ } t\in [0,2a].
\end{equation}
\end{proof}
\begin{rmk}
Under our assumptions, one can obtain $(4)$ directly from the second variation formula. However, our arguments here work for $A^{\partial M}\geq 0$, in which case that the second variation doesn't work.
\end{rmk}

\section{Construction of continuous sweepouts}\label{continuous construction}
In this section, we construct the sweepouts from all free boundary minimal hypersurfaces. First we introduce the continuous min-max theory for compact manifolds with boundary, which is developed by De Lellis-Ramic \cite[\S 1]{DR}. Then we show the foliation of a good neighborhood in Lemma \ref{dist fctn} could be extended to a one-parameter family of hypersurfaces for each orientable minimal hypersurface. In the last part, we construct the sweepout for non-orientable case by considering the orientable 2-sheeted covering.

\subsection{Continuous Sweepouts}
We first introduce some definitions by De Lellis-Ramic \cite[\S 1]{DR}.
\begin{definition}\label{def of sweepout}
A family of $\mathcal H^n$-measurable closed subsets $\{\Gamma_t\}_{t\in[0,1]}$ of $M$ with finite $\mathcal H^n$-measure is called \emph{a generalized smooth family of hypersurfaces with boundary} if
\begin{description}
  \item[$\mathbf{s1}$] for each $t$, there is a finite subset $P_t\subseteq M$ such that $\Gamma_t$ is a smooth hypersurface in $M\setminus P_t$ with boundary in $\partial M\setminus P_t$;
  \item[$\mathbf{s2}$] $t\mapsto\mathcal H^n(\Gamma_t)$ is continuous and $t\mapsto\Gamma_t$ is continuous in the Hausdorff topology;
  \item[$\mathbf{s3}$] $\Gamma_t\rightarrow\Gamma_{t_0}$ smoothly in any compact $U\subset\subset M\setminus P_{t_0}$ as $t\rightarrow t_0$.
\end{description}
A generalized smooth family $\{\Sigma_t\}_{t\in[0,1]}$ is called a \emph{sweepout of $(M,\partial M)$ with portion $T$ (possibly empty)} if there exists a family of relative open sets $\{\Omega_t\}_{t\in[0,1]}$ such that
\begin{description}
  \item[$\mathbf{sw1}$] $(\Sigma_t\setminus\partial\Omega_t)\subseteq P_t$ for any $t\in [0,1]$;
  \item[$\mathbf{sw2}$] Volume($\Omega_t\setminus\Omega_s$)+Volume$(\Omega_s\setminus\Omega_t)\rightarrow 0$ as $s\rightarrow t$;
  \item[$\mathbf{sw3}$] $\Omega_1=M, \Sigma_0=T, \Sigma_t\cap T=\emptyset$ for $t>0$, and  $\{\Sigma_t\}_{t\in[0,\epsilon]}$ is a smooth foliation of a neighborhood of $T$ for some small $\epsilon>0$, i.e. there exists a non-negative Morse function $r$ with
      \begin{equation}
      \Sigma_t= r^{-1}(t) \text{ for } t\in[0,\epsilon],
      \end{equation}
\end{description}
\end{definition}

Let $\{\Gamma_t\}$ be a generalized smooth family of hypersurfaces with boundary, we denote
\begin{equation}
\mathbf L(\{\Gamma_t\})=\sup_t\mathcal H^n(\Gamma_t).
\end{equation}

Let $\{\Gamma_t\}$ and $\{\Gamma_t'\}$ be two sweepouts of $(M,\partial M,g)$, we will say \emph{$\{\Gamma_t\}$ is homotopic to $\{\Gamma'_t\}$} if there exists a 2-parameter family $\{\Psi_{s,t}\}$ such that
\begin{enumerate}
  \item $\{\Psi_{s,\cdot}\}$ is a sweepout for any $s$;
  \item $\{\Psi_{0,t}\}=\{\Gamma_t\}$ and $\{\Psi_{1,t}\}= \{\Gamma'_t\}$;
  \item $\Psi_{s,t}=\Psi_{0,t}$ if $t\in\{0,1\}$.
\end{enumerate}

For a cut-off function $\phi:[0,\infty)\rightarrow [0,1]$ satisfying $\phi(s)=0$ when $s\in[0,\epsilon]\cup[2a-\epsilon,2a]$, there is a family of isotopies of $\tilde M$ generated by the vector field $\phi(\tilde d(x))\nabla\tilde d$. We will use the following proposition in the rest of the section.
\begin{proposition}
Let $(F_t)_{t\in[0,1]}$ be the isotopy of $\tilde M$ generated by $\phi\nabla\tilde d$. Then
\begin{enumerate}
  \item $F_t(x)\cap M'=F_t(x)\cap U$, for any $x\in U$;
  \item for any sweepout $\{\Gamma_t\}_{t\in[0,1]}$ of $(M^+,\partial M^+,\Sigma)$, $\{F_t(\Gamma_t)\cap M^+\}_{t\in[0,1]}$ is also a sweepout, which is homotopic to $\{\Gamma_t\}$.
\end{enumerate}
\end{proposition}
\begin{proof}
Since $F_{-t}(F_t(x))=x$, and $F_t=id$ on $M'\setminus U$, we know that $F_t(x)\in U$. This proved the first one. For the second claim, one can check all the things by Definition \ref{def of sweepout} directly, and we omit it here.
\end{proof}

For a homotopically family $\Lambda$ of sweepouts, we define the \emph{width of $M$ associate with} $\Lambda$ as
\begin{equation}
W(M,\partial M,T,\Lambda)=\inf_{\{\Gamma_t\}\in\Lambda}\mathbf L(\{\Gamma_t\}).
\end{equation}
When $T=\emptyset$, we write the width as $W(M,\partial M,\Lambda)$.
\begin{rmk}
Here $W(M,\partial M,\Lambda)$ is different with the notion in \cite{Zho}. Roughly speaking, in that case, $\Gamma_t$ need to be a closed hypersurface with finite singular points for any $t>0$. However, here $\Gamma_t$ is a smooth hypersurface with boundary in $\partial M$.
\end{rmk}

\subsection{Orientable Case}
We first construct the sweepouts in a small neighborhood of $\Sigma$ in $M^+$ and then make sure it can be extended to the whole $M^+$. The construction here are inspired by Zhou \cite[Proposition 3.6, Proposition 3.8]{Zho} .

By Claim \ref{seperate}, we denote the two components of $M\setminus\Sigma$ as $M^+$ and $M^-$. Let $\n$ be the outward normal vector field of $M^-$ on portion $\Sigma$. Set
\begin{align*}
\mathcal S_+:=\{\Sigma^n:&\Sigma^n \text{ is an embedded orientable connected free } \\
                         &\text{ boundary minimal hypersurface in } M\}.
\end{align*}

Our main purpose of this part is following:
\begin{proposition}\label{ori}
For any $\Sigma\in\mathcal S_+$, there exists a sweepout $\{\Sigma_t\}_{t\in[-1,1]}$ of $M$ such that
\begin{itemize}
  \item $\Sigma_0=\Sigma$;
  \item $\mathcal \area(\Sigma_t)\leq \area(\Sigma)$ with equality only if $t=0$;
  \item $\{\Sigma_t\}_{t\in[-\epsilon,\epsilon]}$ forms a smooth foliation of a neighborhood of $\Sigma$.
\end{itemize}
\end{proposition}

Take $a$ as in Lemma \ref{dist fctn}, then
\begin{equation}
\mathcal \area(d^{-1}(t))<\mathcal \area(\Sigma), \mathrm{\ for\ } t\in(0,2a].
\end{equation}
For any $t\in [0,2a]$, set
\begin{equation}
M^+_{t}:=M^+\setminus\cup_{s\in[0,t]}d^{-1}(s).
\end{equation}
Now we need to extend the local foliation to $M^+_{2a}$. Supposing that for any extended sweepouts $\Lambda$, it always satisfies $\mathbf L(\Lambda)\ge \area(\Sigma)$, then by the following theorem, we can obtain another free boundary minimal hypersurface, which contradicts with Proposition \ref{intersection}.
\begin{theorem}\label{part sweepout}
Let $M'$ and $\Sigma'$ be some $M^+_{\beta}(0<\beta\ll a)$ as well as $\Sigma_s$ above. For any homologically closed family $\Lambda$ of sweepouts of $M'$, with $W(M',\partial M',\Sigma',\Lambda)>\area(\Sigma')$, there exists a min-max sequence $\{\Sigma_{t_n}^n\}$ of $\Lambda$ that converge in the varifold sense to an embedded free boundary minimal hypersurface $\Gamma$ (possibly disconnected), which satisfied $\partial\Gamma\cap\Sigma'=\emptyset$. Furthermore, the width
\begin{equation}
W(M',\partial M',\Sigma',\Lambda)=\area(\Gamma),
\end{equation}
if counted with multiplicities.
\end{theorem}
\begin{proof}
It suffices to show that there exists a minimizing sequence
\begin{equation}
\{\{\Sigma_t^k\}_{t\in[0,1]}\}_{k=1}^\infty\subseteq\Lambda,
\end{equation}
such that
\begin{equation}\label{pull}
\mathcal H^n(\Sigma_t^k)\geq W(M',\partial M',\Sigma',\Lambda)-\delta\Rightarrow d(\Sigma_t^k,\Sigma')\geq\frac{a}{2},
\end{equation}
where $\delta=\frac{1}{4}(W(M',\partial M',\Sigma',\Lambda)-\area(\Sigma'))>0$, and $\mathrm d(\cdot,\cdot)$ is the distance function of $(M,\partial M,g)$.

\begin{lemma}\label{push forward}
For any $\{\Gamma_t\}\in\Lambda$ and $t_0>0$, there exists another $\{\Gamma'_t\}\in\Lambda$ and $\varepsilon\in(0,a)$ satisfying
\begin{enumerate}
  \item $\Gamma_t=\Gamma'_t$ for all $t\in[0,\varepsilon]$;
  \item $\area(\Gamma'_t)\leq\area(\Gamma_t)$ for all $t\in[0,1]$;
  \item\label{away} $\Gamma'_t\subseteq M^+_{\beta+\frac{a}{2}}$ for all $t>t_0$.
\end{enumerate}
\end{lemma}
\begin{proof}
Let $c=\sup_{x\in U}|A(x)|$ (where $A$ is the second fundamental form of the level set of $d$), and $\phi$ be some cut-off function satisfying
\begin{itemize}
  \item $\phi'+c\phi\leq 0$;
  \item $\phi(r)=0$, for all $r>2a$;
\end{itemize}
Denote by $(G_t)_{0\leq t\leq1}$ the one-parameter family of homomorphisms generated by $\phi\nabla\tilde d$. Given surface $L\subseteq M^+_{\beta+2\eta}$ (where $\eta\leq \frac{a}{8}$ and will be identified later) such that $\partial L\subset \partial M\cap M^+_{\beta+2\eta}$, for any $x\in G_t(L)\cap M'$, choose orthonormal basis $\{e_i\}_{i=1}^{n}$ of $T_xG_t(L)$ and satisfies $e_i\perp \nabla d$ for $1\leq i\leq n-1$. Moreover, let $\bar\n$ be the unit normal outward vector field and $e^*$ be the unit vector of the projection of the $\bar \n$ on $T_x(d^{-1}(r(x)))$. By definition,
\begin{align*}
\d_{G_t(L)}(\phi\nabla d)&=\phi'\langle\nabla d,e_n\rangle^2+\phi\d_{G_t(L)}\nabla d\\
                                 &=\phi'\langle\nabla d,e_n\rangle^2+\phi(\d_M\nabla d -\langle\nabla_{\bar\n}X,\bar\n\rangle)\\
                                 &=(\phi'-\phi\langle\nabla_{e^*}\nabla d,e^*\rangle)\langle\bar\n,e^*\rangle^2+\phi \d_M\nabla d\\
                                 &\leq (\phi'-\phi\langle\nabla_{e^*}\nabla d,e^*\rangle)\langle\bar\n,e^*\rangle^2\\
                                 &\leq 0.
\end{align*}

Now take $\eta\leq \frac{a}{8}$ small enough such that
$\Gamma_t\subseteq M_{\beta+2\eta}$ for $t\in[\frac{t_0}{2},1]$. As $\{G_s(L)\cap U\}_{s\in[0,t]}$ is the free variation (in the sense of \S \ref{two variations}) of $L$ by the vector field $\phi\nabla d$,  applied Corollary \ref{variation} directly,
\begin{align*}
&\ \ \ \ \frac{d}{ds}\area(G_s(L)\cap M')\\
&=\int_{G_s(L)\cap U}\d_{G_s(L)}(\phi\nabla d)-\int_{\partial (G_s(L)\cap M')}\frac{\langle\phi\nabla d,\nu_{\partial M}\rangle}{|(\nu_{\partial M})^\top|}\\
&\leq 0.
\end{align*}
This implies
\begin{equation}\label{area decrea}
\area(G_t(L)\cap M')\leq \area(G_0(L)\cap M')=\area(L),\ \ \forall L\subseteq M_{\beta+2\eta}.
\end{equation}
Now let $S>0$ be such that $G_S(d^{-1}(\beta+2\eta))\cap M'=d^{-1}(\beta+\frac{a}{2})$ and then choose a smooth non-negative function $h:[0,1]\rightarrow [0,S]$ such that $h(t)=0$ for $t<\frac{t_0}{2}$ and $h(t)=S$ for $t\geq t_0$. Set
\begin{equation}
\Gamma'_t=G_{h(t)}(\Gamma_t)\cap M'.
\end{equation}
Then if $t\leq \frac{t_0}{2}$, $\Gamma'_t=\Gamma$; if $t\geq\frac{t_0}{2}$, it follows from the definition of $t_0$ that $\Gamma_t\subseteq M_{2\eta}$ and then $\area(G_{h(t)}(\Gamma_t))\leq \area(\Gamma_t)$ by (\ref{area decrea}). For the last requirement, first notice that $h(t)=S$ if $t\geq t_0$, then by combining the results $G_S(d^{-1}(\beta+2\eta))\cap M'=d^{-1}(\beta+\frac{a}{2})$ with $\Sigma_t\subseteq M_{\beta+2\eta}$, we conclude that $\Sigma'_t\subseteq M_{\beta+\frac{a}{2}}$. This completes the proof of the lemma.
\end{proof}

We can now finish the argument. For any $\{\Gamma_t^k\}\in\Lambda$, there always exists $\epsilon_k>0$ such that
\begin{equation}\label{local depart}
\mathcal H^n(\Gamma_t^k)\leq\area(\Sigma')+\delta\ \ \text{ for all } t\in[0,2\epsilon_k].
\end{equation}
Then take $t_0=\epsilon_k$ in the lemma above, we can obtain a better sweepout $\{\Sigma_t^k\}$, which will satisfies (\ref{pull}). In fact,
\begin{equation}
\mathcal H^n(\Sigma^k_t)\geq W(M',\partial M',\Sigma',\Lambda)-\delta,
\end{equation}
implies
\begin{equation}
\mathcal H^n(\Gamma^k_t)\geq W(M',\partial M',\Sigma',\Lambda)-\delta=\area(\Sigma')+\delta
\end{equation}
and then by (\ref{local depart}), we have $t\geq 2\epsilon_k$. Now use Lemma \ref{push forward} (\ref{away}), we obtain
\begin{equation}
d(\Sigma_t^k,\Sigma')\geq\frac{a}{2}.
\end{equation}

Now modifying the arguments of min-max theory for compact manifold with boundary in \cite{DR}, we can get a free boundary minimal surface $(\Gamma,\partial\Gamma)$  with $\partial\Gamma\subseteq\partial M$. Let us sketch the main steps here.

Let $\{\{\Sigma_t^n\}_{t\in[0,1]}\}_{n=1}^\infty$ be the minimizing sequence. First we follow the tightening process by De Lellis-Ramic \cite[Proposition 3.2]{DR}, where we deform each $\{\Sigma_t\}_{t\in[0,1]}$ to another one $\{\tilde\Sigma_t\}_{t\in[0,1]}$ such that every min-max sequence $\{\tilde\Sigma_{t_k}^k\}$ converges to a stationary varifold. Since those $\Sigma_t^k$ with volume close to $W(M',\partial M',\Sigma')$ have a distance $a/2>0$ away from $\Sigma'$, we can take $\tilde\Sigma_t^k=\Sigma_t^k$ near $\Sigma'$. Hence $\{\tilde\Sigma_t^k\}$ can be chosen to satisfy (\ref{pull}).

Now for an almost minimizing min-max sequence $\{\tilde\Sigma_{t_k}^k\}$ (see \cite[Proposition 4.3]{DR}), it follows that $\tilde\Sigma_t^k$ always have a distance $a/2$ away from $\Sigma'$ for large $k$ by (\ref{pull}). Hence all the Definition 4.1 and Proposition 4.3 in \cite{DR} are well-defined.

Finally we show that the limit of the almost minimizing min-max sequence is supported on some embedded free boundary minimal hypersurface. These were done by De Lellis-Ramic \cite[\S 10.3, \S 10.4]{DR}. There are no differences here. Hence we can get a free boundary minimal hypersurface $(\Gamma,\partial\Gamma)$  with $\partial\Gamma\subseteq\partial M$. Since the minimizing sequence have fixed distance to $\Sigma'$, we conclude that $\Gamma\cap\Sigma'=\emptyset$. However this contradicts with following Frankel's property by Fraser-Li.
\end{proof}

\begin{lemma}[\cite{FL} Lemma 2.5]\label{intersection}
Let $(M,\partial M,g)$ be a connected, compact manifold with convex boundary and $\ric\geq0$, then any two properly embedded connected free boundary minimal hypersurfaces $\Sigma$ and $\Sigma'$ must intersect.
\end{lemma}

Now we can obtain the sweepout from orientable free boundary minimal hypersurface:
\begin{proof}[Proof of Proposition \ref{ori}]
By function $r$ in the Lemma \ref{dist fctn}, we can get a sweepout in a neighborhood $U$ of $\Sigma$. Moreover, $M$ will be separated into $M^+$ and $M^-$ by $\Sigma$. The boundary of $M^{+}_{\epsilon}$ has two parts: $\Sigma_{\epsilon}$ and $\partial M\cap M^+_{\epsilon}$. By the Theorem \ref{part sweepout}, there is a sweepout $\{\Sigma'_t\}$ of $(M^+_{\epsilon},\partial M^+_{\epsilon},\Sigma_\epsilon)$ such that
\begin{enumerate}
  \item $\Sigma'_t=\Sigma_t+\epsilon$ for $0\leq t\leq\varepsilon$;
  \item $\area(\Sigma'_t)\leq\area(\Sigma_\epsilon)\leq\area(\Sigma)$.
\end{enumerate}
We can also construct an sweepout for $(M^-_{\epsilon},\partial M^-_{\epsilon},\Sigma_\epsilon)$ by the same way, and then patch them all together to we get a sweepout of $M$ which satisfies all the requirements in the Proposition \ref{ori}.
\end{proof}

\subsection{Non-orientable Case}\label{Non-orientable}
For the non-orientable case, $\Sigma$ won't separate $M$, otherwise $\Sigma$ would be part of $\partial(M\setminus\Sigma)$, which must be orientable. Hence $\tilde M=M\setminus\Sigma$ would be a connected compact manifold with piecewise smooth boundary. One part is $\partial M\setminus\partial\Sigma$ and the other is $\tilde\Sigma$, which is the double cover of $\Sigma$. Since $\Sigma$ is a free boundary minimal surface, two parts of $\partial\tilde M$ will meet orthogonally. Now take two $\tilde M$ and patch them together by identifying two $\tilde\Sigma$. Denote the new manifold by $\bar M$. Then $\bar M$ is the double cover of $M$. More importantly, $\tilde\Sigma$ is an orientable free boundary minimal hypersurface and hence we can use the sweepout above to get the the sweepouts here. Set
\begin{align*}
\mathcal S_-:=\{\Sigma^n:&\Sigma^n \text{ is an embedded non-orientable connected free}\\
                         &\text{ boundary minimal hypersurface in } M\}.
\end{align*}
\begin{proposition}\label{nonori}
For any $\Sigma\in\mathcal S_-$, there exists a family $\{\Sigma_t\}_{t\in[0,1]}$ of closed sets of $M$ such that
\begin{itemize}
  \item $\Sigma_0=\Sigma$;
  \item $\{\Sigma_t\}$ satisfies $(s1)(sw1)(sw2)(sw3)$ in Definition \ref{def of sweepout};
  \item $\max\mathcal H^n(\Sigma_t)=2\area(\Sigma)$ and $\mathcal H^n(\Sigma_t)<2\area(\Sigma)$;
  \item in $(s2)$, only fails when $t\rightarrow 0$, $\mathcal H^n(\Sigma_t)\rightarrow 2\area(\Sigma)$;
  \item in $(s3)$, only fails when $t\rightarrow 0$, $\Sigma_t\rightarrow 2\Sigma$.
\end{itemize}
\end{proposition}
\begin{proof}
Consider about the double cover $\bar M$ and construct the sweepout of the orientable manifold $(\bar M,\partial M)$. In order to define the sweepout of $M$, we can identify $M\setminus\Sigma$ with a component of $\bar M\setminus\bar\Sigma$. Finally, define $\Sigma_0=\emptyset$. One can check all the requirements in Proposition \ref{nonori}.
\end{proof}

\section{Almgren-Pitts discrete setting for manifolds with boundary}\label{discrete settings}
Recently, Li-Zhou \cite{LZ} developed the Almgren-Pitts min-max theory for any compact manifold with boundary. In this section, we give a brief introduction to these theory. For the basic notations in geometric measure theory, we refer to \cites{Pit, Alm62, MN14}. The following homotopy relations were introduced in \cite[\S 4.1]{Pit}. We refer to \cite[\S 5.1]{LZ} for the case of compact manifolds with boundary, which we focus on here.

Let $(M^{n+1},\partial M,g)$ be some Riemannian manifold with convex boundary and $2\leq n\leq 6$. We assume more that $(M,\partial M,g)$ is embedded in some $\R^N$ for some $N$ large enough. Let us denote $\mathbf I_k(M)$ be the space of $k$-dimensional integral currents with support in $M$, and
\begin{gather*}
\mathbf I_k (M,\partial M):=\{T:T\in\mathbf I_k(M),\mathrm{spt}(\partial T)\subseteq\partial M\},\\
\mathcal Z_k(\partial M)=\{T: T\in\mathbf I_k(\partial M), \partial T=0\},\\
Z_k(M,\partial M)=\{T: T\in\mathbf I_k(M,\partial M), \mathrm{spt}(\partial T)\in\partial M\}.
\end{gather*}
We will say $T$ and $S$ are in the same equivalent class if $T,S\in Z_k(M,\partial M)$ and $T-S\in \mathbf I_k(\partial M)$. We denote $\mathcal Z_k(M,\partial M)$ as all the equivalent class in $Z_k(M,\partial M)$ and $\pi: Z_k(M,\partial M)\rightarrow \mathcal Z_k(M,\partial M)$ as the projection. Moreover, for any $T\in \mathcal Z_k(M,\partial M)$, there is a \emph{canonical representation} $\zeta(T)\in Z_k(M,\partial M)$ such that $\zeta(T)\llcorner\partial M=0$.

Given any $T\in\mathbf I_k(M,\partial M)$, let $|T|, \Vert T\Vert$ be the integral varifold and Radon measure associated with $T$ respectively. Given any surface $\Sigma$ with possible non-empty boundary or open set $\Omega\subseteq M$, we denote $\llbracket \Sigma\rrbracket ,\llbracket \Omega\rrbracket $, and $[\Sigma],[\Omega]$ as the integral currents and integral varifold, respectively.

We also need the metrics on these spaces. Let $\mathbf M$ be the mass norm on $\mathbf I_k(M)$ and $\mathcal F$ the flat metric on it. In the space of relative cycles, the flat metric and mass norm are defined to be
\begin{gather*}
\mathcal F(P,Q)=\inf\{\mathcal F(S+R,T):S\in P,T\in Q,R\in\mathbf I_k(M)\},\\
\mathbf M(P)=\inf\{\mathbf M(T+R):R\in\mathbf I_k(\partial M)\}.
\end{gather*}
If we use the standard representation of $P\in\mathcal Z_k(M,\partial M)$
In the following of the papar, we will focus on the $1$-sweepout, hence the notations about cell complex will be restricted to this case.
\begin{definition}[\cite{Zho} Definition $4.1$]Set $I=[0,1]$.
\begin{itemize}
  \item The $0$-complex $I_0=\{[0],[1]\}$;
  \item For any $i\in\N$, $I(1,j)$ has $0$-complex $\{[\frac{i}{3^j}]\}$ and $1$-complex $\{[\frac{i}{3^j},\frac{i+1}{3^j}]\}$. We always denote $I(1,j)_p$ the set of $p$-complex of $I(1,j)$;
  \item Given $\alpha\in I(1,j)_1$, we denote $\alpha(k)_p$ as the $p$-complex of $I(1,j+k)$ contained in $\alpha$;
  \item The boundary homeomorphism $\partial: I(1,j)_1\rightarrow I(1,j)_0$ is $\partial[a,b]=[b]-[a]$;
  \item The distance function $d:I(1,j)_0\times I(1,j)_0\rightarrow \mathbb\N$ is $d(x,y)=3^j|x-y|$.
\end{itemize}
\end{definition}
\begin{definition}[Fineness]
For any $\phi:I(m,j)_0\rightarrow\mathcal Z_n(M,\partial M)$, the $\mathbf M$-\emph{fineness of} $\phi$ is
\begin{equation}
\mathbf f_{\mathbf M}(\phi):=\sup\big\{\frac{\mathbf M(\phi(x)-\phi(y))}{d(x,y)}:x,y\in I(m,j)_0,x\neq y\big\}.
\end{equation}
\end{definition}

\begin{definition}[Homotopy for mappings]
Let $\phi_i:I(1,j_i)_0\rightarrow \mathcal Z_n(M,\partial M)$ for $i=1,2$ and $\delta>0$, we say $\phi_1$ \emph{is} $1$-\emph{homotopic to} $\phi_2$ \emph{with $\mathbf M$-fineness} $\delta$ if there exists $j_3>j_1,j_2$ and
\begin{equation*}
\psi: I(1,j_3)_0\times I(1,j_3)_0\rightarrow \mathcal Z_n(M,\partial M),
\end{equation*}
with
\begin{itemize}
  \item $\mathbf f_{\mathbf M}(\psi)\leq\delta$;
  \item $\psi(i-1,x)=\phi_i(n(j_3,j_i)(x)),i=1,2$;
  \item $\psi(I(1,j_3)_0\times I_0(i,j_3)_0)=0$.
\end{itemize}
\end{definition}

\begin{definition}
For a sequence of
\begin{equation*}
\phi_i:I(1,j_i)_0\rightarrow\mathcal Z_n(M,\partial M),
\end{equation*}
$\{\phi_i\}_{i\in \mathbb N}$ is a $(1,\mathbf M)$-\emph{homotopy sequence of mappings into} $(\mathcal Z_n(M,\partial M),\{0\})$ if $\phi_i$ is $1$-homotopic to $\phi_{i+1}$ with fineness $\delta_i\rightarrow 0$, and
\begin{equation}
\sup_{i}\{\mathbf M(\phi_i(x)):x\in\dmn\phi_i\}<\infty.
\end{equation}
\end{definition}

\begin{definition}[Homotopy for sequence of mappings]
Let $S_1=\{\phi_i^1\}_{i\in\N}$ and $S_2=\{\phi_i^2\}_{i\in\N}$ be two $(1,\mathbf M)$-homotopy sequence of mappings into $(\mathcal Z_n(M,\partial M),\{0\})$, we say $S_1$ is homotopic to $S_2$ if $\phi_i^1$ is $1$-homotopic to $\phi_i^2$ with fineness $\delta_i\rightarrow 0$.
\end{definition}

Denote $\pi^\sharp_1(\mathcal Z_n(M,\partial M,\mathbf M),\{0\})$ the space of all equivalent classes of $(1,\mathbf M)$-homotopy sequences of mappings into $(\mathcal Z_n(M,\partial M),\{0\})$. Similarly, we can define $\pi^\sharp_1(\mathcal Z_n(M,\partial M,\mathcal F),\{0\})$. By \cite[Throrem 4.6]{Pit}, these two homotopy groups are isomorphic, furthermore, they are both homotopic to $H_n(M,\partial M)$.

Let $\Pi\in\pi^\sharp_1(\mathcal Z_n(M,\partial M),\{0\})$, then for any $S=\{\phi_i\}\in\Pi$, we define
\begin{equation}
\mathbf L(S)=\limsup_{i\rightarrow\infty}\max_{x\in\dmn\phi_i}\mathbf M(\phi_i(x)),
\end{equation}
and the \emph{width of} $\Pi$
\begin{equation}
\mathbf L(\Pi)=\inf_{S\in\Pi}\mathbf L(S).
\end{equation}

In \cite{LZ}, Martin Li and Xin Zhou proved the following min-max theorem:
\begin{theorem}[\cite{LZ} Theorem 5.21 and Theorem 6.2]
For any homotopy class $\Pi\in\pi^\sharp_1(\mathcal Z_n(M,\partial M,\mathbf M),\{0\})$, there exists an integral varifold $V$ such that
\begin{itemize}
  \item $\Vert V\Vert (M)=\mathbf L(\Pi)$;
  \item $V$ is almost minimizing in small annuli with free boundary;
  \item $V=\sum n_i[\Sigma_i]$ where $n_i\in\N$ and each $(\Sigma_i,\partial\Sigma_i)\subseteq(M,\partial M)$ is a smooth compact connected embedded free boundary minimal hypersurface.
\end{itemize}
\end{theorem}

\section{Discretization}\label{Discretization}
In this section, we discretize the continuous sweepouts in Section \ref{continuous construction} to the Almgren-Pitts setting. We will use the following Discretization Theorem by Li-Zhou \cite{LZ}:
\begin{theorem}[\cite{LZ} Theorem 5.12]\label{discrete}
Given a map
\begin{equation*}
\Phi: I^m\longrightarrow \Z_n(M,\partial M),
\end{equation*}
which is continuous in the $\mathcal F$-topology and satisfying the following:
\begin{itemize}
  \item $\sup_{x\in I^m}\mathbf M(\Phi(x))<\infty$;
  \item $\lim_{r\rightarrow 0}\m(\Phi,r)=0$;
  \item $\Phi|_{I_0^m}$ is continuous in the $\F$-metric,
\end{itemize}
then there exists a sequence of mappings
\begin{equation*}
\phi_i: I(m,j_i)_0\longrightarrow \Z_n(M,\partial M),
\end{equation*}
with $j_i<j_{i+1}$ and a sequence of positive numbers $\delta_i\rightarrow 0$ such that
\begin{enumerate}
  \item $S=\{\phi_i\}$ is an $(m,M)$-homotopy sequence with $\mathbf M$-fineness
  $\mathbf f_{\mathbf M}(\phi_i)<\delta_i$;
  \item There exists a sequence of $k_i$ such that for all $x\in I(m,j_i)_0$,
  \begin{equation*}
  \mathbf M(\phi_i(x))\leq \sup\{\mathbf M(\Phi(y)):\alpha\in I(m,k_i)_m, x, y\in\alpha\}+\delta_i.
  \end{equation*}
  In particular, we have $\mathbf L(S)\leq\sup_{x\in I^m}\mathbf M(\Phi(x))$.
  \item $\sup\{\mathcal F(\phi_i(x)-\Phi(x)):x\in I(m,j_i)_0\}<\delta_i$;
  \item $\mathbf M(\phi_i(x))<\mathbf M(\Phi(x))+\delta_i$ for all $x\in I(m,j_i)_0$.
\end{enumerate}
\end{theorem}

Here we only need the case of $m=1$ and then the third requirement in Theorem \ref{discrete} is trivial. Also, we always have
\begin{equation}
\sup_{x\in[0,1]}\mathbf M(\Phi(x))<\infty,
\end{equation}
by Proposition \ref{ori} and Proposition \ref{nonori}. For the second requirement, we need to show the function
\begin{align*}
f:[0,r_0]\times M\times[0,1]\rightarrow\R^+,\\
(r,p,x)\mapsto \Vert \Sigma_x\Vert (B_r(p)),
\end{align*}
is continuous. There are no differences with Lemma 5.3 in \cite{Zho} so we omit it here.

In order to prove the final result, we need to show that all the discrete families are corresponding to fundamental class in $H_{n+1}(M,\partial M)$. The idea here is inspired by Zhou \cite[Theorem 5.8]{Zho}. The difference is that we have boundary terms here. However, we show that all the boundary terms in $\mathbf I_n(\partial M)$, and then Constancy Theorem (see \cite[\S 26.27]{Sim}) works here.
\begin{theorem}\label{homotopy}
Given a continuous sweepout as in Proposition \ref{ori} and Proposition \ref{nonori}, and $\{\phi_i\}_{i\in\N}$ the corresponding $(1,\mathbf M)$-homotopy sequence obtained by Theorem \ref{discrete}, assume that $\Phi(x)=\llbracket \partial\Omega_x\rrbracket ,\forall x\in[0,1]$ where $\{\Omega_t\}_{t\in[0,1]}$ is a family of open sets satisfying $(sw2)(sw3)$ in Definition \ref{def of sweepout}. If $F:\pi_1^\sharp(\mathcal Z_n(M^{n+1},\partial M,\mathbf M),\{0\})\rightarrow H_{n+1}(M^{n+1},\partial M,\mathbb Z)$ is the isomorphism given by Almgren \cite[\S3]{Alm62}, then
\begin{equation}
F([\{\phi_i\}_{i\in\N}])=\llbracket M\rrbracket,
\end{equation}
where $\llbracket M\rrbracket $ is the fundamental class of $M$.
\end{theorem}
\begin{proof}
First we review the isomorphism $F:\pi_1^\sharp(\mathcal Z_n(M^{n+1},\partial M,\mathbf M),\{0\})\rightarrow H_{n+1}(M^{n+1},\partial M,\mathbb Z)$ by Almgren \cite[\S 3]{Alm62}. Take $\phi_i: I(1,j_i)_0\rightarrow\mathcal Z_n(M^{n+1},\partial M)$ to be the map constructed in \cite[Theorem 5.12]{LZ}. For any 1-cell $\beta\in I(1,j_i)_1$ with $\beta=[t_\beta^1,t_\beta^2]$, $\mathcal F(\phi_i(t_\beta^1),\phi_i(t_\beta^2))\leq\mathbf M(\phi_i(t_\beta^1),\phi_i(t_\beta^2))\leq\mathbf f_\mathbf M(\phi_i)\leq\delta_i$. Then by the $\mathbf M$-isoperimetric lemma \cite[Lemma 4.15]{LZ}, there exists an isoperimetric choice $Q_\beta\in \mathbf I_{n+1}(M^{n+1})$ with
\begin{itemize}
  \item $\partial Q_\beta=\zeta(\phi_i(t_\beta^2))-\zeta(\phi_i(t_\beta^1))+R_\beta$, for some $R_\beta\in\mathbf I_n(\partial M)$;
  \item $\mathbf M(Q_\beta)+\mathbf M(R_\beta)\leq C_M\mathcal F(\phi_i(t_\beta^1),\phi_i(t_\beta^2))$.
\end{itemize}
Then $F$ is defined by Almgren \cite[\S 3]{Alm62}:
\begin{equation}
F([\{\phi_i\}_{i\in \N}])=\sum_{\beta\in I(1,j_i)_1}\llbracket Q_\beta\rrbracket .
\end{equation}
Recall the construction of discretization in \cite[Theorem 5.12]{LZ} (see also \cite[Theorem 13.1]{MN14}): there exists $k_i, l_i>0$ such that $j_i=k_i+l_i+1$ and
\begin{itemize}
  \item $\phi_i([\frac{s}{3^{k_i}}])=\Phi(\frac{s}{3^{k_i}})=\pi(\llbracket \partial\Omega_{\frac{s}{3^{k_i}}}\rrbracket)$ in $\mathcal Z_n(M^{n+1},\partial M)$ for any positive integer $s$, that is, there exists $R\in\mathbf I_n(\partial M)$ such that
  \begin{equation}
      \phi_i([\frac{s}{3^{k_i}}])=\llbracket \partial\Omega_{\frac{s}{3^{k_i}}}\rrbracket+R;
  \end{equation}
  \item $\mathcal F(\phi_i(\frac{s}{3^{k_i}}),\phi_i(\frac{s+1}{3^{k_i}}))\leq\delta_i$;
\end{itemize}
For 1-cell $\alpha\in I(1,k_i)$, set
\begin{equation}
\tilde F(\alpha,\phi_i)=\sum_{\beta\in\alpha(l_i+1)_1}\llbracket Q_\beta\rrbracket .
\end{equation}
\begin{claim}
For any 1-cell $\alpha_j=[\frac{j}{3^{k_i}},\frac{j+1}{3^{k_i}}]$,
\begin{equation}
\tilde F(\alpha_j,\phi_i)=\llbracket \Omega_{\frac{j+1}{3^{k_i}}}\rrbracket -\llbracket \Omega_{\frac{j}{3^{k_i}}}\rrbracket ,
\end{equation}
in $\mathbf I_{n+1}(M^{n+1})$.
\end{claim}
Hence
\begin{equation}
F([\{\phi_i\}_{i\in\N}])=\sum_{\alpha\in I(1,k_i)_1}\tilde F(\alpha,\phi_i)=\llbracket M\rrbracket .
\end{equation}
To complete the proof of the theorem, it suffices to prove the claim. By direct computation,
\begin{align*}
\partial\tilde F(\alpha_j,\phi_i)&=\sum_{\beta\in\alpha_j(l_i+1)_1}\partial\llbracket Q_\beta\rrbracket \\
                                      &=\sum_{\beta\in\alpha_j(l_i+1)_1}\zeta(\phi_i(t_\beta^2))-\zeta(\phi_i(t_\beta^1))+R_\beta\\
                                      &=\zeta(\phi_i(\frac{j+1}{3^{k_i}}))-\zeta(\phi_i(\frac{j}{3^{k_i}}))+\sum_{\beta\in\alpha_j(l_i+1)_1} R_\beta\\
                                      &=\partial\llbracket\Omega_{\frac{j+1}{3^{k_i}}}\rrbracket-\partial\llbracket\Omega_{\frac{j}{3^{k_i}}}\rrbracket+R'\\
                                      &=\partial\llbracket\Omega_{\frac{j+1}{3^{k_i}}}-\Omega_{\frac{j}{3^{k_i}}}\rrbracket+R',
\end{align*}
for some $R'\in I_n(\partial M)$. Then by the Constancy Theorem \cite[Theorem 26.27]{Sim}, we obtain
\begin{equation}
\tilde F(\alpha_j,\phi_i)-\llbracket \Omega_{\frac{j+1}{3^{k_i}}}-\Omega_{\frac{j}{3^{k_i}}}\rrbracket =k\llbracket M\rrbracket ,
\end{equation}
for some $k\in\mathbb Z$. So we only need to show the mass of the left hand side is small, which comes from the construction of $\phi_i$ like the proof of \cite[Theorem 5.8]{LZ}  and we omit it here.
\end{proof}

Above all, given any $\Sigma\in\mathcal S_+\cup S_-$, let $\Phi^{\Sigma}:[0,1]\rightarrow(\mathcal Z_n(M^{n+1},\partial M,\{0\})$ be the continuous sweepout given by Proposition \ref{ori} and Proposition \ref{nonori}. Then apply the Theorem \ref{homotopy} to get a $(1,\mathbf M)$-homotopy sequence $\{\phi_i^{\Sigma}\}_{i\in\N}$ into $(\Z_n(M^{n+1},\partial M,\mathcal F),\{0\})$ and
\begin{itemize}
  \item They are all in the same homotopy class and $F([\{\phi_i\}])=\llbracket M\rrbracket $;
  \item $\mathbf L([\{\phi_i\}])\leq \area(\Sigma)$ if $\Sigma\in\mathcal S_+$;
  \item $\mathbf L([\{\phi_i\}])\leq 2\area(\Sigma)$ if $\Sigma\in\mathcal S_-$.
\end{itemize}

\section{Multiplicity and Orientation of non-orientable part}\label{Multiplicity and Orientation}
In this section, we show that the min-max minimal hypersurface corresponding to the fundamental class $[M]$ is orientable. In the first part, we show that the multiplicity of the non-orientable part is even. Then since we have good sweepout with multiplicity 2 for non-orientable hypersurface, we know the multiplicity of non-orientable part can only be 2 (see Section \ref{main proof} for more details). In the second part, we show that the non-orientable minimal hypersurfaces can not be produced by min-max theory. To show this, we construct a better one-family sweepout which is in the same homotopic class and has width less than double of the area of the non-orientable hypersurface. The two parts are inspired by Zhou \cite[Proposition 6.1]{Zho} and Ketover-Marques-Neves \cite{KMN}. For completeness of this paper, we put the details in Appendix and sketch the steps here.
\subsection{Multiplicity}
In this part, we discuss the multiplicity of the min-max free boundary minimal hypersurfaces. By the min-max theory for the compact manifolds with boundary, the stationary varifold is an integer multiple of some smooth minimal free boundary minimal hypersurface (denoted it by $\Sigma$).
\begin{proposition}\label{even multiple}
Let $\Sigma$ be the stationary varifold in Theorem \ref{main thm}, with $\Sigma=\cup_{i=1}^lk_i[\Sigma_i]$, where $\{\Sigma_i\}$ is a disjoint collection of smooth connected embedded free boundary minimal hypersurfaces with multiplicity $k_i\in\N$. If $\Sigma_i$ is non-orientable, then $k_i$ must be an even number.
\end{proposition}
\begin{proof}[Outline of Proof of Proposition \ref{even multiple}]
We sketch the steps here and put the details in Appendix B. By Li-Zhou \cite[Theorem 5.21]{LZ}, $\Sigma$ is almost minimizing in small annuli with free boundary (with Definition 5.19 in \cite{LZ}). Let $\Sigma_1$ be a non-orientable component of $\Sigma$. Now taking $p$ in the interior of $\Sigma_1$, and $r>0$, we can find $T_i\in\mathcal Z_k(M,\partial M)$ such that
\begin{itemize}
  \item $T_i\llcorner B(p,r)$ is locally mass minimizing in $B(p,r)$;
  \item $\lim_{i\rightarrow\infty}|T_i|=|\Sigma|$ as varifold.
\end{itemize}
Then by Compactness Theorem for relative cycles \cite[Lemma 4.10]{LZ}, $T_i$ converges to $T_0\in\mathcal Z_k(M,\partial M)$ up to subsequence. $T_0$ here is in fact a integral cycle in $\cup_{i=1}^l\Sigma_i$, and hence the coefficient of the non-orientable part is even. Last, we shrink $r$ such that $\{T_i\}$ have bounded first variation to use White's Theorem \cite{Whi}. For more details, see Appendix.
\end{proof}

\subsection{Rule Out The Non-orientable Case}
In \cite[Theorem 3.5 and Theorem 4.1]{KMN}, Ketover-Marques-Neves ruled out the non-orientable part in closed manifolds by Catenoid estimates. Supposing that non-orientable hypersurface $\Sigma$ with multiplicity two is the min-max minimal hypersurface corresponding to the fundamental class, then one can always amend the sweepouts by add tubes to reduce the width of the sweepouts, contradicting with the assumptions. Recently, Haslhofer-Ketover \cite[\S 4]{HK} applied Catenoid estimate to give an upper bound of 2-width. Moreover, the idea still works for compact manifolds with assumptions in Theorem \ref{main thm}. We clarify the proposition below and put the constructions of sweepouts in Appendix C for completeness of this paper.

\begin{proposition}\label{rule out}
For any $(M^{n+1},\partial M,g)$ with convex boundary, and $3\leq n+1\leq 7$, then the min-max minimal hypersurface corresponding to the fundamental class $[M]$ can not be a non-orientable hypersurface with multiplicity $2$.
\end{proposition}
\begin{proof}
Suppose that $\Sigma$ is non-orientable and $2\Sigma$ is the min-max free boundary minimal hypersurface corresponding to the fundamental class $[M]$ in $H_{n+1}(M,\partial M,\mathbb Z)$. Notice that $\Sigma$ does not separate $M$. Let $\bar M$ be the double cover of $M$, $\tau:\bar M\rightarrow\bar M$ is the covering map,  $\{\bar\Sigma_t\}_{t\in[-1,1]}$ is the sweepout of $\bar M$ in Proposition \ref{ori}. Now by the second variation of the area formula,
\begin{equation*}
\area(\bar\Sigma_s)=2\area(\Sigma)-\frac{s^2}{2}(\int_{\bar\Sigma_0}(|A|^2+\ric(\n,\n))+\int_{\partial\Sigma_0}h(\n,\n)) +O(s^3).
\end{equation*}
Hence for $s$ small enough, there exists $A>0$ so that
\begin{equation}
\mathcal H^n(\bar\Sigma_s)\leq 2\area(\Sigma)-As^2.
\end{equation}

In Proposition \ref{nonori}, we constructed the sweepout $\{\Sigma_t\}_{t\in[0,1]}$ of $M$ from sweepout $\{\bar\Sigma_t\}_{t\in[0,1]}$ of $\bar M^+$ in Proposition \ref{ori}. We now will amend the sweepout $\{\bar\Sigma_t\}_{t\in[0,1]}$ above to produce a sweepout $\{\Lambda'_t\}_{t\in[0,1]}$ of $M$ which satisfies $\mathbf L(\{\Lambda'_t\})<\mathbf L(\{\Sigma_t\})$. By above area formula, for any $\delta>0$, there exists $\epsilon>0$ so that
\begin{equation}\label{depart estimate}
\sup_{t\in[\delta,1]}\mathcal H^n(\bar\Sigma_t)\leq\mathcal H^n(\bar\Sigma_0)-\epsilon.
\end{equation}
Similar to the proof in \cite{KMN},
\begin{enumerate}
  \item $\Lambda_t=\bar\Sigma_t$ for $t>\delta$;
  \item for $t\in[0,\delta]$, $\Lambda_t$ open up via cylinders (or Catenoid for $n=2$) at some points.
\end{enumerate}
By Appendix C, the amended sweepout $\{\Lambda_t'\}$ satisfies
\begin{equation}
    \mathcal H^n(\Lambda'_t)\leq 2\area(\Sigma)-\frac{A}{2}\alpha^2-cr^2, t\in [0,\delta],
\end{equation}
where $\alpha$ is the height of cylinder or catenoid, $r$ is the radius of intersection of cylinder and $\Sigma$, $A,c$ are constants. Taking $\epsilon'=\min\{\epsilon,A\alpha^2,2cr^2\}$, then
\begin{equation}
    \mathcal H^n(\Lambda'_t)\leq 2\area(\Sigma)-\frac{\epsilon'}{2}.
\end{equation}
This contradicts with the choice of $\Sigma$. We put more details in Appendix C.
\end{proof}

\section{Proof of the main theorem}\label{main proof}
Now we can prove our main results:
\begin{proof}[Proof of Theorem \ref{main thm}]
For any $\Sigma\in\mathcal S_+\cup\mathcal S_-$, take $\Phi^{\Sigma}$ as the continuous sweepouts given by Proposition \ref{ori} and Proposition \ref{nonori} and let $S_{\Sigma}=\{\phi_i^{\Sigma}\}_{i\in\N}$ be the corresponding $(1,\mathbf M)$-homotopy sequence. After Theorem \ref{homotopy}, we summarized that all $S_{\Sigma}$ lie in the same homotopy class $F^{-1}(\llbracket M\rrbracket )$. Let us denote this homology class by $\Pi_M$, then
\begin{equation}
\mathbf L(\Pi_M)\leq\mathcal A(M,\partial M),
\end{equation}
where $\mathcal A(M,\partial M)$ is defined by
\begin{equation}
\mathcal A(M,\partial M):=\inf(\{\area(\Sigma)|\ \Sigma\in S_+\}\cup\{2\area(\Sigma)|\ \Sigma\in S_-\}).
\end{equation}
By min-max Theorem for compact manifold with boundary developed by Li-Zhou \cite[Theorem 5.21, Theorem 6.2]{LZ}, there exists a stationary integral varifold $\Sigma$, which supported on a free boundary minimal hypersurface $\Sigma_0$, such that $\mathbf L(\Pi_M)=\Vert \Sigma\Vert (M)$. Notice that $\Sigma_0$ must be connected since $(M,\partial M)$ has positive Ricci curvature and convex boundary. Hence $\Sigma=k\llbracket \Sigma_0\rrbracket $ for some $k\in\N$, $k\neq 0$. By the definition of $\mathcal A(M,\partial M)$,
\begin{itemize}
  \item if $\Sigma_0\in\mathcal S_+$, then $k\leq 1$ and hence $k=1$, $\area(\Sigma_0)=\mathcal A(M,\partial M)$;
  \item if $\Sigma_0\in\mathcal S_-$, then $k\leq 2$ and must be even by Proposition \ref{even multiple}, hence $k=2$ and $\mathcal A(M,\partial M)\leq 2\area(\Sigma_0)\leq \mathcal A(M,\partial M)$, this implies $\mathcal A(M,\partial M)=2\area(\Sigma_0)$.
\end{itemize}
However, by Theorem \ref{rule out}, the second case can not happen. Hence we have proved the min-max minimal hypersurface corresponding to the fundamental class is orientable with multiplicity one and $\area(\Sigma_0)=\mathcal A(M,\partial M)$. Now the only thing we need to show is that $\text{index}(\Sigma_0)=1$. We will use the same arguments with \cite{MN14} and \cite[Claim 5]{Zho}. Let $\{\Sigma_t\}_{t\in[-1,1]}$ be the sweepout which we constructed in Proposition \ref{ori}, then there is a family of diffeomorphisms of $\tilde M$ corresponding to $X\in\mathfrak X(\tilde M)$ such that
\begin{itemize}
  \item $X|_{\Sigma_0}$ is the normal vector field;
  \item $\Sigma_t=F_t(\Sigma_0)\cap M$ for $t\in[-\epsilon,\epsilon]$.
\end{itemize}

Supposing that $\text{index}(\Sigma_0)\geq 2$, then there exists a function $u\neq0$ such that $Q(1,u)=0$, and $Q(u,u)<0$ where
\begin{equation}
Q(f,g)=\int_{\Sigma_0}\langle\nabla^{\Sigma_0}f,\nabla^{\Sigma_0}g\rangle-(|A^{\Sigma_0}|^2+\ric_M(\n,\n))fg- \int_{\partial\Sigma_0}h^{\partial M}(\n,\n)fg,
\end{equation}
$\n$ is the unit normal vector field of $\Sigma_0$ in $M$, $A$ is the corresponding second fundamental form, and $h$ is the second fundamental form on $\partial M$ in $M$. Here $h^{\partial M}(\n,\n)>0$ since $\partial M$ is convex.

Let $\tilde X\in\mathfrak X(\tilde M)$ be the extension of $u\n$ and $\{\tilde F_s\}_{s\in[-\epsilon',\epsilon']}$ be the corresponding family of diffeomorphisms of $\tilde M$. Set $\Sigma_{s,t}=\tilde F_s(\Sigma_t)$ and $\tilde f(s,t) =\mathcal H^n(\Sigma_{s,t})$. Then $\nabla\tilde f(0,0)=0$ since $\Sigma_0$ is stationary. Furthermore,
\begin{gather*}
\frac{\partial^2}{\partial s\partial t}\tilde f(0,0)=Q(1,u)=0,\\
\frac{\partial^2}{\partial s^2}\tilde f(0,0)=Q(u,u)<0,\\
\frac{\partial^2}{\partial t^2}\tilde f(0,0)=Q(1,1)<0.
\end{gather*}
So there exists $\delta>0$ such that $\tilde f(\delta,t)<f(0,0)$ for all $t\in[-1,1]$. By Theorem \ref{discrete} and Theorem \ref{homotopy}, we can construct a $(1,\mathbf M)$-homotopy sequence $\{\phi_i^{\delta}\}_{i\in\N}\in\Pi_M$ such that
\begin{equation}
\mathbf L(\{\phi_i^{\delta}\})\leq\sup_{t\in[-1,1]}\tilde f(\delta,t)<\tilde f(0,0)=\area(\Sigma_0)=\mathcal A(M,\partial M),
\end{equation}
which contradicts with $\mathbf L(\Pi_M)=\mathcal A(M,\partial M)$. Hence $\text{index}(\Sigma_0)=1$.
\end{proof}

\vspace{2em}
\appendix
\renewcommand{\appendixname}{Appendix~\Alph{section}}

\section*{Appendix A: Non-existence of closed minimal hypersurfaces}\label{app non-exist of closed}
We give the proof of the claim in Remark \ref{remark1}:
\begin{proof}[Proof of the statement in Remark \ref{remark1}]
Supposing that $\Sigma\subset M$ is a closed minimal hypersurface in $M$. Let $x\in\Sigma$ and $y\in\partial M$ such that
\begin{equation}
d(:=d(x,y))=d(\Sigma,\partial M),
\end{equation}
and $\gamma(t)$ is the distance geodesic between $x$ and $y$. It follows that $\gamma'(0)\perp\Sigma$ and $\gamma'(d)\perp\partial M$. Let $\{e_i\}_{i=1}^n\cup \{\gamma'(0)\}$ be the orthonormal basis of $T_x M$, then extending them to the vector fields on $\gamma(t)$ by parallel moving, which still denoted by $\{e_i\}$. Now computing the second variation of these vector fields:
\begin{equation}
0\leq\delta^2L(\gamma)(e_i,e_i)\leq-\int_{\gamma}R(\gamma'(t),e_i,e_i,\gamma'(t))dt-A^{\partial M}(e_i,e_i)-A^{\Sigma}(e_i,e_i),
\end{equation}
then taking the sum,
\begin{equation}
0\leq -\int \ric(\gamma'(t),\gamma'(t))-H^{\partial M},
\end{equation}
which contradicts with the non-positive Ricci curvature and convex boundary conditions.
\end{proof}

\vspace{1em}
\section*{Appendix B: Even multiplicity of non-orientable part}\label{app even multiplicity}
\begin{proof}[Proof of Proposition \ref{even multiple}]
By Theorem $5.21$ in \cite{LZ}, $\Sigma$ is almost minimizing in small annuli with free boundary (with Definition 5.19 in \cite{LZ}). Let $\Sigma_1$ be a non-orientable component of $\Sigma$. Now take $p$ in the interior of $\Sigma_1$, and $r>0$ such that
\begin{itemize}
  \item $B(p,2r)$ is contained in some $A(p',s,r_{p'})$;
  \item $\Sigma$ is almost minimizing in $A(p',s,r_{p'})$;
  \item $\mathrm{spt}(\Vert \Sigma\Vert )\cap B(p,2r)=\mathrm{spt}(\Vert \Sigma_1\Vert )\cap B(p,2r)$ is diffeomorphic to an $n$-ball.
\end{itemize}
Hence by \cite{LZ} Proposition 6.3 , for any $K\subseteq B(p,r)$, there exists varifold $V^*$ (called replacement of $\Sigma$ in $K$), and a sequence of $T_i\in\mathcal Z_k(M,\partial M)$ such that
\begin{itemize}
  \item $T_i\llcorner B(p,r)$ is locally mass minimizing in $B(p,r)$;
  \item $\lim_{i\rightarrow\infty}|T_i|=V^*$ as varifold.
\end{itemize}
Attention that in \cite{LZ} Theorem $6.3$, $T_i\in\mathcal Z_k(M,M\setminus B(p,2r))$ and the convergence is in the sense of $\mathbf F_{B(p,2r)}$-metric. However, by step 2 in the proof of \cite[Theorem 6.3]{LZ}, we can see our statements here is also true.

Since the Lemma $6.7$, Lemma $6.9$, Lemma $6,11$ in \cite{Zho} are local propositions, it still holds in this case, that is, we still have $V^*=|\Sigma|$ as varifold. Now since $T_i\in\mathcal Z_k(M,\partial M)$ and $\mathbf M(T_i)$ are uniformly bounded, then the Compactness Theorem for relative cycles (see \cite{LZ} Lemma 4.10) implies there exists a subsequence, still denoted by $\{T_i\}$, converges to some $T_0\in\mathcal Z_k(M,\partial M)$. However, the associated varifolds $|T_i|$ converge to $[\Sigma]$, hence $\mathrm{spt} T_0\subseteq \cup_{i=1}^l \Sigma_i$. Moreover, we have
\begin{claim}
$T_0$ is a relative integral cycle in $\cup_{i=1}^l\Sigma_i$.
\end{claim}
By the Constancy Theorem (\cite{Sim} Theorem 26.27), $T_0=\sum_{i=1}^lk_i'\llbracket \Sigma_i\rrbracket ,$ for some $k_i'\in\mathbb Z$. The lower semi-continuity of the mass implies $|k_i'|\leq k_i$. The key point here is that $k_1'$ must be even, or $k_1'\Sigma_1$ can not represent a relative integral cycle.

Now we focus on the ball $B(p,s)$. We can shrink the radius slightly such that $\partial(T_i\llcorner B(p,s))$ have uniformly bounded mass (by slicing theory, \cite{Sim} 28.5), and hence in the sense of subsequence, converge to some limit current. Since $T_i$ are locally mass minimizing in $B(p,r)$, we can apply the Theorem 6.12 in \cite{Whi} by White,
\begin{equation}
\Sigma\llcorner B(p,s)=[T_0\llcorner B(p,s)]+2W,
\end{equation}
hence $k_1$ is even.
\end{proof}

\vspace{1em}
\section*{Appendix C: Construction of new sweepout}\label{app new sweepout}
\begin{proof}[Construction in Proposition \ref{rule out}]
Now we amend the sweepouts in two different cases:

The First Part : $3\leq n\leq 6$.

We will open up via cylinders in $\bar M$ and make them invariant under the covering deformation and hence it will be a sweepout of $M$. For any $p\in\bar\Sigma_0$, let
\begin{gather*}
\mathcal C_{r,\alpha}(p):=\{\text{exp}_x(t\n(x)):x\in\partial B_r(p)\cap\bar\Sigma_0,t\in[-\alpha,\alpha]\},\\
B_{r,\alpha}:=\{\text{exp}_x(\alpha\n(x)):x\in B_r(p)\cap\bar\Sigma_0\}
\end{gather*}
Now for fixed $p\in\bar\Sigma_0,R>0, \alpha>0$, there exists $c,C>0$ so that for any $r<R, |\alpha'|<\alpha$,
\begin{gather*}
c|\alpha'| r^{n-1}\leq\mathcal H^n(\mathcal C_{r,\alpha}(p))\leq C|\alpha'| r^{n-1},\\
cr^n\leq\mathcal H^n(B_{r,\alpha}(p))\leq Cr^n.
\end{gather*}
Now let us denote
\begin{equation}
\Lambda_{r,\alpha}:=(\bar\Sigma_\alpha\setminus(B_{r,\pm\alpha}(p)\cup B_{r,\pm\alpha}(\tau(p))))\cup\mathcal C_{r,\alpha}(p)\cup\mathcal C_{r,\alpha}(\tau(p)),
\end{equation}
It follows that
\begin{equation}
\mathcal H^n(\Lambda_{r,\alpha})\leq 2\mathcal H^n(\bar\Sigma_0)+2C\alpha r^{n-1}-2c r^n-A\alpha^2.
\end{equation}
Since $n>2$, we can shrink $\alpha\ll1$(only depends on $C,c$ and $A$) such that for all $\alpha'\leq \alpha$,
\begin{equation}
2C\alpha' r^{n-1}-c r^n\leq\frac{A}{2}\alpha'^2,
\end{equation}
hence,
\begin{equation}
\mathcal H^n(\Lambda_{r,\alpha'})\leq 2\mathcal H^n(\bar\Sigma_0)-c r^n-\frac{A}{2}\alpha'^2.
\end{equation}
For fixed $\alpha$ above, let $R<\alpha$, define the sweepout
\begin{equation}
\Lambda_t=\left\{\begin{aligned}
         &\Lambda_{R,\frac{2t\alpha}{\delta}},&t\in\big(0,\frac{\delta}{2}\big]\ \\
         &\Lambda_{2R-\frac{2Rt}{\delta},\alpha},\ &t\in\big(\frac{\delta}{2},\delta\big].
         \end{aligned}\right.
\end{equation}
By the estimates above,
\begin{gather*}
\mathcal H^n(\Lambda_{R,\beta})\leq2\mathcal H^n(\bar\Sigma_0)-cR^n,\\
\mathcal H^n(\Lambda_{r,\alpha})\leq2\mathcal H^n(\bar\Sigma_0)-\frac{A}{2}\alpha^2,
\end{gather*}
If we choose $\bar\epsilon=\min\{\epsilon,cR^n,\frac{A}{2}\alpha^2\}$(where $\epsilon$ is the constant in equation (\ref{depart estimate})), we obtain that for all $t\in[0,1]$,
\begin{equation}
\mathcal H^n(\Lambda_t)\leq2\mathcal H^n(\bar\Sigma_0)-\bar\epsilon.
\end{equation}
Since $\Lambda_t$ are all invariant under the covering deformation, we can get a sweepout $\{\Lambda'_t\}$ of $M$ by quotient the covering deformation like what we did in Proposition \ref{nonori}, and the new sweepout would satisfies
\begin{equation}
\mathcal H^n(\Lambda'_t)\leq 2\area(\Sigma)-\frac{\bar\epsilon}{2}.
\end{equation}
This contradicts with the assumption of $\Sigma$.

\

The Second Part: $n=2$.

We will use the following logarithmically cut-off function near $p$:
\begin{equation}
\eta_{r,R}(x)=\left\{\begin{aligned}
          &1 &d(x,p)\geq R\\
          &(\log r-\log d(x,p))/(\log (r)-\log R) &r\leq d(x,p)\leq R\\
          &0 &d(x,p)\leq r
          \end{aligned}\right.
\end{equation}
Instead of the cylinder in the case of $n\geq 3$, we will use the following surfaces which are like the catenoid:
\begin{equation}
\mathcal D_{r,R}^\alpha(p):=\{\text{exp}_{x}(\pm\alpha\eta_{r,R}\n(x)):x\in \big(B_R(p)\setminus B_r(p)\big)\cap \bar\Sigma_0\}.
\end{equation}
Then define the surface
\begin{equation}
\Lambda_{r,R}^\alpha=\big(\Sigma_{\pm\alpha}\setminus\big(B_{R,\pm\alpha}(p)\cup B_{R,\pm\alpha}(\tau(p))\big)\big)\cup\mathcal D_{r,R}^{\pm\alpha}(p)\cup\mathcal D_{r,R}^{\pm\alpha}(\tau(p)).
\end{equation}
By \cite{KMN} Proposition 2.5,
\begin{align*}
\mathcal H^2(\mathcal D_{r,R}^\alpha(p))\leq &2\mathcal H^2(B_R(p)\setminus B_r(p))+\alpha^2\int_{B_R(p)\setminus B_r(p)} |\nabla\eta_r|^2\\
      &+C\alpha^3\int_{B_R(p)}(1+(|\nabla\eta|^2)),
\end{align*}
hence
\begin{align*}
\mathcal H^2(\Lambda_{r,R}^\alpha)\leq &\mathcal H^2(\bar\Sigma_{\pm\alpha}\setminus B_{R,\pm\alpha}(p))+2\mathcal H^2(B_R(p)\setminus B_r(p))\\
&+2\mathcal H^2(B_R(p)\setminus B_r(\tau(p)))+\frac{D\alpha^2}{\log(R/r)}.
\end{align*}
We can choose $\bar R$ small enough so that there exists $A,c,C$, for $R\leq\bar R$,
\begin{gather*}
\mathcal H^2(\bar\Sigma_{\pm\alpha}\setminus B_{R,\pm\alpha}(p))\leq 2\mathcal H^2(\bar\Sigma_0\setminus B_R(p))-A\alpha^2,\\
cR^2\leq\mathcal H^2(B_R(p))\leq CR^2.
\end{gather*}
It follows that
\begin{equation}
\mathcal H^2(\Lambda_{r,R}^\alpha)\leq 2\mathcal H^2(\bar\Sigma_0)-2A\alpha^2-2cr^2+\frac{D\alpha^2}{\log(R/r)}.
\end{equation}
Now fix $R,r\leq \bar R$ such that
\begin{equation}
\frac{D}{\log(R/r)}\leq A,
\end{equation}
then we have the area estimate
\begin{equation}
\mathcal H^2(\Lambda_{r,R}^\alpha)\leq 2\mathcal H^2(\bar\Sigma_0)-A\alpha^2-2cr^2.
\end{equation}
Now define the sweepout
\begin{equation}
\Lambda_t=\left\{\begin{aligned}
          &\Lambda_{r,R}^{2t\alpha/\delta} &t\in(0,\frac{\delta}{2}]\\
          &\Lambda_{\frac{2rt}{\delta}-2r,\frac{2Rt}{\delta}-2R}^\alpha &t\in(\frac{\delta}{2},\delta].
          \end{aligned}\right.
\end{equation}
Take $\epsilon'=\min\{\epsilon,A\alpha^2,2cr^2\}$, then
\begin{equation}
\mathcal H^2(\Lambda_t)\leq 2\mathcal H^2(\bar\Sigma_0)-\epsilon',
\end{equation}
Since the sweepout is invariant under the covering deformation, we can construct the sweepout $\{\Lambda'_t\}$ of $M$ by quotient the covering deformation, and we will have the estimate
\begin{equation}
\mathcal H^2(\Lambda'_t)\leq 2\mathcal H^2(\Sigma)-\frac{\epsilon'}{2},
\end{equation}
\end{proof}

\bibliographystyle{amsbook}

\end{document}